\documentclass[a4paper,12pt]{amsart}

\usepackage[english]{babel}
\usepackage[alphabetic,lite]{amsrefs}
\usepackage{adjustbox}

\usepackage[cmtip,arrow]{xy}
\usepackage{pb-diagram,pb-xy}

\usepackage{etoolbox}
\apptocmd{\sloppy}{\hbadness 10000\relax}{}{}

\newtheorem{theorem}{Theorem}[section]
\newtheorem{lemma}[theorem]{Lemma}
\newtheorem{proposition}[theorem]{Proposition}
\newtheorem{cor}[theorem]{Corollary}
\newtheorem{question}[theorem]{Question}
\newtheorem{conjecture}[theorem]{Conjecture}
\theoremstyle{definition}
\newtheorem{definition}[theorem]{Definition}
\newtheorem{example}[theorem]{Example}
\theoremstyle{remark}
\newtheorem{remark}[theorem]{Remark}

\newcommand{\NN}{\mathbb{N}} %natural numbers
\newcommand{\RR}{\mathbb{R}} %real numbers
\newcommand{\CC}{\mathbb{C}} %complex numbers
\newcommand{\mat}[3]{{\mathrm{Mat}\left( #1 \times #2 ; \, #3 \right)}} %space of matrices
\newcommand{\udisk}{\mathbb{D}} %unit disc
\newcommand{\sdisk}{\mathbb{G}} %symmetrized polydisc
\newcommand{\slgrp}{\mathrm{SL}} %the group SL
\newcommand{\glgrp}{\mathrm{GL}} %the group GL
\newcommand{\slalg}{\mathfrak{sl}} %the Lie algebra sl
\newcommand{\glalg}{\mathfrak{gl}} %the Lie algebra gl
\newcommand{\sball}{\Omega} %spectral ball
\newcommand{\aut}[1]{\mathrm{Aut} \left( #1 \right)} %group of automorphisms
\newcommand{\faut}[1]{\mathrm{Aut}^{\pi}\!\! \left( #1 \right)} %group of automorphisms
\newcommand{\cfaut}[1]{\mathrm{Aut}^{\pi}_0\! \left( #1 \right)} %group of automorphisms, connected component
\newcommand{\hol}{\mathcal{O}} %holomorphic structure sheaf
\newcommand{\boundary}{\partial} %boundary
\newcommand{\id}{\mathrm{id}}
\newcommand{\cont}{\mathcal{C}}
\title[The fibred density property and the spectral ball]{The fibred density property and the automorphism group of the spectral ball}
\author{Rafael~B.~Andrist}
\address{Rafael~B.~Andrist \\ School of Mathematics and Natural Sciences \\ University of Wuppertal \\ Germany}
\email{rafael.andrist@math.uni-wuppertal.de}

\author{Frank~Kutzschebauch}
\address{Frank~Kutzschebauch \\ Institute of Mathematics \\ \hfill University of Bern \\ Switzerland}
\email{frank.kutzschebauch@math.unibe.ch}


\begin{document}

\begin{abstract}
We generalize the notion of the density property for complex manifolds to holomorphic fibrations, and introduce the notion of the fibred density property. We prove that the natural fibration of the spectral ball over the symmetrized polydisc enjoys the fibred density property and describe the automorphism group of the spectral ball.
\end{abstract}

\maketitle

\section{Introduction and Results}

The density property for complex manifolds has been introduced by Varolin \cites{Varolin1,Varolin2} in 2000 and has its orgin in the work of Anders\'en \cite{AL1990} and Anders\'en--Lempert \cite{AL1992} in the early 1990s.
It characterizes complex manifolds with large automorphism groups. Since then the so-called Anders\'en--Lempert theory has developed rapidly. It has many applications for geometric questions in Several Complex Variables and has contributed a lot to a better understanding of large holomorphic automorphism groups. For a recent overview of the theory we refer to \cite{densitysurvey}.

In this paper we introduce a parametrized version of the density property for holomorphic fibrations where the density property holds only in the direction of possibly singular fibres which are parametrized by the base space. For a comparison we restate the definition of the density property.

\begin{definition}
Let $X$ be a complex manifold. We say that $X$ has the \emph{density property} if the Lie algebra generated by all holomorphic $\CC$-complete vector fields on $X$ is dense (w.r.t.\ compact-open topology) in the Lie algebra of all holomorphic vector fields on $X$.
\end{definition}

\begin{definition}
We call a holomorphic surjection $\pi \colon X \to Y$ between complex manifolds a \emph{holomorphic fibration}. We say that the fibration has the \emph{fibred density property} if the Lie algebra generated by all holomorphic $\CC$-complete vector fields on $X$ tangent to the fibres of $\pi$ is dense (w.r.t.\ compact-open topology) in the Lie algebra of all holomorphic vector fields on $X$ tangent to $\pi$.
\end{definition}

\begin{example}
\label{trivialexample}
The fibred density property is known for a trivial fibration: consider the projection $\pi \colon U \times V \to V$ for Stein manifolds $U$ and $V$ where $U$ has the density property. Then this fibration has the fibred density property, as a consequence of \cite{Varolin1}*{Lemma 3.5}. For the case $\CC^k \times \CC^n \to \CC^k$ with fibers $\CC^n, n \geq 2,$ see also \cite{ALpara}*{Corollary 2.2}.
\end{example}

\smallskip

In this paper we will first develop the general theory for the fibred density property (section \ref{sectionAL}) and then prove the fibred density property for a well-known quotient map from classical invariant theory. This enables us to describe the automorphism group of the spectral ball.

\bigskip

Concerning the general theory the following theorem is a fibred version of the classical Anders\'en--Lempert theorem \cites{AL1990, AL1992, Varolin1, Varolin2}, for a special case see also  \cite{ALpara}*{Theorem 2.3}.

\begin{theorem}
\label{thmfibredAL}
Let $X$ be a Stein manifold and let $\pi \colon X \to Y$ be a holomorphic fibration with the fibred density property. Let $\Omega \subseteq X$ be an open subset and $\varphi_t \colon \Omega \to X, \; t \in [0,1],$ be a fibre-preserving $\cont^1$-homotopy of injective holomorphic maps such that $\varphi_0$ is the natural embedding $\Omega \hookrightarrow X$ and $\varphi_t(\Omega)$ is Runge in $X$ for all $t \in [0,1]$. Then there exists a fibre-preserving homotopy $\Phi_t \colon X \to X$ of holomorphic automorphisms such that $\Phi_0 = \id_X$ and $\Phi_t$ is arbitrarily close to $\varphi_t$ on $\Omega$ in the compact-open topology.

Moreover, $\Phi_t$ can be chosen as a composition of flow maps corresponding to complete fibre-preserving vector fields which generate a dense Lie subalgebra of the Lie algebra of the fibre-preserving holomorphic vector fields.
\end{theorem}

By a \emph{fibre-preserving homotopy} $\varphi_t \colon \Omega \to X$ we mean a homotopy such that $\pi \circ \varphi_i = \pi$. A vector field $\Theta$ is called \emph{fibre-preserving} if it is tangential to the fibres, i.e.\ $d\pi(\Theta) = 0$.
 
By $\faut{X}$ we denote the group of fibre-preserving holomorphic automorphisms of $X$, i.e.\
\begin{equation*}
\faut{X} := \left\{ f \colon X \to X \; \text{holomorphic automorphism}, \; \pi \circ f = \pi \right\}
\end{equation*}
and by $\cfaut{X}$ its path-connected component of identity.

By choosing $\Omega = X$ and $\varphi_t$ a path in $\faut{X}$ starting at $\id_X$, we obtain:
\begin{cor}
\label{thmfibredALcor}
Let $X$ be a Stein manifold and let $\pi \colon X \to Y$ be a holomorphic fibration with the fibred density property.
Then all fibre-preserving holomorphic automorphisms $\cfaut{X}$ path-connected to the identity can be approximated by compositions of time-$1$ maps of flows of complete fibre-preserving vector fields which generate a dense Lie subalgebra of the Lie algebra of fibre-preserving holomorphic vector fields.
\end{cor}

\bigskip

Next we present an application of our general result concerning the fibred density to a concrete example, the so-called spectral ball. 
It appears naturally in Control Theory \cites{robustcontrol-proc, robustcontrol}, but is also of theoretical interest in Several Complex Variables.

\begin{definition}
The \emph{spectral ball} of dimension $n \in \NN$ is defined to be
\[
\sball_n := \{ A \in \mat{n}{n}{\CC} \;:\; \rho(A) < 1 \}
\]
where $\rho$ denotes the spectral radius, i.e.\ the modulus of the largest eigenvalue.
\end{definition}

The study of the group of holomorphic automorphisms of the spectral ball started with the work of Ransford and White \cite{RanWhi1991} in 1991 and continued by various authors \cites{BariRans,Rostand,Thomas}. Recently, Kosi{\'n}ski \cite{Kosinski2013} described a dense subgroup of the $2 \times 2$ spectral ball $\Omega_2$. We generalize this result to $\Omega_n$ for $n \geq 2$ with another approach using the fibred density property.

The spectral ball $\sball_n$ can also be understood in the following way: Denote by $\sigma_1, \dots, \sigma_n \colon \CC^n \to \CC$ the elementary symmetric polynomials in $n$ complex variables. Let $ \mathrm{Eig} \colon \mat{n}{n}{\CC} \to \CC^n$ assign to each matrix a vector of its eigenvalues.
Then we denote by $\pi_1 := \sigma_1 \circ \mathrm{Eig}, \dots, \pi_n := \sigma_n \circ \mathrm{Eig}$ the elementary symmetric polynomials in the eigenvalues. By symmetrizing we avoid any ambiguities of the order of eigenvalues in the definition of $\mathrm{Eig}$ and obtain a polynomial map $\pi_1, \dots, \pi_n$, symmetric in the entries of matrices in $\mat{n}{n}{\CC}$, actually
\[
\chi_A(\lambda) = \lambda^n + \sum_{j=1}^n (-1)^j \cdot \pi_j(A) \cdot \lambda^{n-j}
\]
where $\chi_A$ denotes the characteristic polynomial of $A$.

Now we can consider the fibration $\pi := (\pi_1, \dots, \pi_n) \colon \Omega_n \to \sdisk_n$ of the spectral ball over the symmetrized polydisc $\sdisk_n := (\sigma_1, \dots, \sigma_n)(\udisk^n)$. A generic fibre, i.e.\ a fibre above a base point with no multiple eigenvalues, consists exactly of one equivalence class of similar matrices. Therefore it is natural to study the action of $\slgrp_n(\CC)$ on $\sball_n$ by conjugation.

A generic fibre is obviously a homogeneous space and hence smooth. A fibre above a base point with multiple eigenvalues decomposes into several strata of $\slgrp_n(\CC)$ orbits where the largest orbit is the orbit of a matrix with the largest possible Jordan blocks. The structure of these fibres is well-known in classical invariant theory, see e.g.\ \cite{Kraft}.

Because a generic fibre is a homogeneous space of the complex Lie group $\slgrp_n(\CC)$, it has the density property (according to  \cite{DoDvKa2010}). This does however not imply the fibred density property, since the dependence on the base point and the role of singular fibres is a priori not clear. However it motivates the investigation of fibre-preserving automorphisms by exploiting the homogeneity of the generic fibres. The most difficult part of our paper is to prove the \emph{fibred density property} for $\pi \colon \sball_n \to \sdisk_n$ (see Theorem \ref{fibredforspectral}) that enables us to determine a dense subgroup of the holomorphic automorphism group $\aut{\sball_n}$.

\begin{theorem}
\label{thmspectraldense}
The $\slgrp_n(\CC)$-shears and the $\slgrp_n(\CC)$-overshears together with matrix transposition and M\"obius transformations generate a dense subgroup (in compact-open topology) of the holomorphic automorphism group $\aut{\sball_n}$.
\end{theorem}

The precise definitions of shears and overshears will be given in section \ref{sectionshear}. They are obtained as time-$1$ maps of certain re-parametrizations of flow maps of complete vector fields; in case of $\slgrp_n(\CC)$-shears and $\slgrp_n(\CC)$-overshears they are certain re-parametrizations of $1$-parameter subgroups of $\slgrp_n(\CC)$ which act on $\mat{n}{n}{\CC}$ by conjugation.
For the definition of M\"obius transformations of matrices we refer to equation \eqref{eqmoebius} on page \pageref{eqmoebius}.

\smallskip
Similar to the situation for the holomorphic automorphism group of $\CC^n, \; n \geq 2,$ it seems impossible to give an explicit set of algebraic generators for $\aut{\sball_n}$, since we prove the following in the last section:
\begin{theorem}
\label{thmmeagre}
The dense subgroup of $\aut{\sball_n}$ generated by the automorphisms in Theorem \ref{thmspectraldense} is a meagre subset of $\aut{\sball_n}$.
\end{theorem}

Since many homogeneous spaces of complex Lie groups enjoy the density property, it is natural to ask the following question.

\begin{question}
For which holomorphic actions of a reductive group $G$ on a Stein manifold $X$ does
the map $\pi \colon X \to X/\!/G$ to the categorical quotient admit the fibred density property?
\end{question}

Acknowledgement: The authors would like to thank the referee for recommendations to improve the presentation and for pointing out some problems in the first version of this article.
Moreover they would like to thank A.~Liendo, P.-M.~Poloni, S.~Maubach, A.~van~den~Essen, H.~Derksen and A.~Nowicki for discussions about determining the dimension
growth of the kernel of a homogeneous derivation, which led to the formulation of  Conjecture \ref{derivationconjecture}.

\section{Anders\'en--Lempert theory for fibrations}
\label{sectionAL}

We follow the original idea of the Anders\'en--Lempert Theorem, see \cites{AL1990,AL1992} and also the survey article \cite{densitysurvey} and the textbook \cite{Forstneric-book}*{Sec. 4} for a more recent presentation.

\begin{definition}[\cite{flowstuff}*{p. 254} and \cite{Forstneric-book}*{Def. 4.8.1}]
Let $\Theta$ be a vector field on a complex manifold $X$, and let $(t, x) \mapsto A_t(x)$ be a continuous map to $X$, defined on an open subset of $\RR \times X$ containing $\{0\} \times X$ such that its $t$-derivative exists and is continuous. We say that A is \emph{algorithm} for $\Theta$ if we have for all $x \in X$ that
\begin{align*}
A_0(x) &= x\\
\left.\frac{\mathrm{d}}{\mathrm{d} t}\right|_{t=0} A_t(x) &= \Theta_x
\end{align*}
\end{definition}

Obviously, a flow map is always an algorithm whereas the converse does not need to be true. However, the following variant of Euler's method for solving an ODE works:

\begin{proposition}[\cite{flowstuff}*{Thm. 4.1.26} and \cite{Forstneric-book}*{Thm. 4.8.2}]
\label{propalgos}
Let $\Theta$ be a locally Lipschitz continuous vector field with flow $\varphi_t$ on a complex manifold $X$. Let $\Omega$ be the fundamental domain of $\Theta$ and $\Omega_+ := \Omega \cap (\RR_+ \times X)$. If $A_t$ is an algorithm for $\Theta$, then for all $(t, x) \in \Omega_+$ the $n$-th iterate $A^n_{t/n}(x)$ of the map $A_{t/n}$ is defined for sufficiently large $n = n(t, x) \in \NN$ and we have
\[
\lim_{n \to \infty} A^n_{t/n}(x) = \varphi_t(x)
\]
The convergence is uniform on compacts in $\Omega_+$. Conversely, if $t_0 > 0$ is such that $A^n_{t/n}(x)$ is defined for all $t \in [0, t_0]$ and all sufficiently large $n \in \NN$, and $\lim_{n \to \infty} A^n_{t/n}(x)$ exists, then $(t_0, x) \in \Omega_+$.
\end{proposition}

For simplicity, we focus on the situation $t \geq 0$, but the same results hold for negative times by replacing $\Theta$ with $-\Theta$. The following lemma can be verified easily in local coordinates by Taylor series expansion, see e.g.\ \cite{Forstneric-book}*{Prop. 4.7.3}.
\begin{lemma}
\label{lemalgos}
Let $X$ be a complex manifold and let $\Theta$ and $\Xi$ be holomorphic vector fields on $X$ with flow maps or algorithms $\varphi_t$ and $\psi_t$. Then
\begin{enumerate}
\item $\varphi_t \circ \psi_t$ is an algorithm for $\Theta + \Xi$
\item $\psi_{-\sqrt{t}} \circ \varphi_{-\sqrt{t}} \circ \psi_{\sqrt{t}} \circ \varphi_{\sqrt{t}}$ is an algorithm for $[\Theta, \Xi]$.
\end{enumerate}
\end{lemma}

\begin{proof}[Proof of Theorem \ref{thmfibredAL}]
By $\mathcal{A}$ we denote a dense Lie subalgebra of complete fibre-preserving vector fields which is dense in the Lie algebra of fibre-preserving holomorphic vector fields. This Lie subalgebra exists by assumption (fibred density property).

We define a time-dependent vector field 
\[
\Theta_z^t := \dot{\varphi_t}(\varphi_t^{-1}(z))
\]
which is still tangent to the fibres of $\pi \colon X \to Y$.
For any $n \in \NN$ we can partition the interval $[0, 1]$ in $n$ intervals $[k/n, (k+1)/n]$ of length $1/n$ and consider the piecewise constant vector field
\[
\widehat{\Theta}_z^t := \Theta_z^{k/n} \; \text{for } t \in [k/n, (k+1)/n) 
\]

Let $\varphi_t^k$ denote the flow map of $\Theta_z^{k/n}$.
Because $\Omega \subseteq X$ is Runge and $X$ is Stein, we can approximate any fibre-preserving vector field on $\Omega$ by a fibre-preserving vector field on $X$, uniformly on compacts of $\Omega$. We remark that the sheaf of germs of fibre-preserving vector fields is (as the kernel of the map induced by $\pi$ between the coherent sheafs of sections of the tangent bundles) a coherent sheaf of $\hol(X)$ modules. A standard application of Cartan's Theorems A and B implies that the sections of any coherent sheaf over a Runge subset in a Stein space can be approximated by global sections.

By assumption we know that every such vector field can be approximated by vector fields from the Lie algebra $\mathcal A$. Proposition \ref{propalgos} and Lemma \ref{lemalgos} then show that we are able to approximate the flows of all the vector fields in the closure of $\mathcal A$ in the compact-open topology by the flows of the complete vector fields which generate $\mathcal A$.

The composition of the flows $\varphi_{\dots}^k \circ \dots \circ \varphi_{\dots}^1$ is the flow of $\widehat{\Theta}_z^t$. It remains only to show that in the limit $n \to \infty$ the flow of $\widehat{\Theta}^t$ converges uniformly on compacts to the flow of $\Theta^t$ which follows from the fact that these flows are tangent to each other at the times $k/n$.
\end{proof}

\section{Shears and Overshears for $\slgrp_n(\CC)$}
\label{sectionshear}

As a preparation for proving a fibred density property for the spectral ball, we need to study a special type of automorphisms, the so-called shears and overshears. They will serve us as building blocks for general automorphisms.

The following notion of generalized shears and overshears has been introduced by Varolin \cite{shearinvent}*{Section 3}.
\begin{definition}
Let $X$ be a complex manifold and let $\Theta$ be a $\CC$-complete vector field on $X$, i.e.\ such that its flow-map exists for all complex times. A vector field $f \cdot \Theta, \; f \in \hol(X),$ is called a \emph{$\Theta$-shear vector field} if $\Theta(f) = 0$. It is called a \emph{$\Theta$-overshear vector field} if $\Theta^2(f) = 0$. 
\end{definition}

\begin{example}
Let $X = \CC^2$ with coordinates $(z,w)$. Then $\Theta = \partial_{z}$ is obviously a $\CC$-complete vector field, with flow map $\phi_t(z,w) = (z+t,w)$. A $\partial_{z}$-shear vector field is of the form $f(w) \cdot \partial_{z}$ and a $\partial_{z}$-overshear vector field is of the form $(f(w) \cdot z + g(w)) \cdot \partial_{z}$ where $f,g \in \hol(\CC)$.
\end{example}

The following $\CC$-completeness result can be found also \cite{shearinvent}*{Section 3}, but without an explicit formula for the flow map. Our proof gives an explicit formula which will be needed in the applications.
\begin{lemma}
\label{lemshearflow}
Let $X$ be a complex manifold and let $\Theta$ be a $\CC$-complete vector field on $X$, then all $\Theta$-overshear vector fields are $\CC$-complete as well. In fact, if $\phi_t$ denotes the flow map of $\Theta$, the flow map $\psi_t$ of $f \cdot \Theta$ is  given by
\begin{align*}
\psi_t(z) &= \phi_{\displaystyle \varepsilon(t \Theta_z f) \cdot t f(z)}(z) 
\end{align*}
where $\varepsilon \colon \CC \to \CC$ is given by
\begin{align*}
\varepsilon(\zeta) &= \sum_{k=1}^\infty \frac{\zeta^{k-1}}{k!} = \frac{e^\zeta - 1}{\zeta}
\end{align*}
\end{lemma}
\begin{remark}
The flow map of a $\Theta$-shear takes the form
\begin{align*}
\psi_t(z) &= \phi_{t f(z)}(z) 
\end{align*}
In particular, if $\phi_t$ and $f$ are polynomial, then $\psi_t$ is polynomial as well.
\end{remark} 

\begin{proof} %\hfill
%Let $\phi_t$ denote the flow map of $\Theta$, i.e.\
%\begin{equation*}
%\frac{\mathrm{d}}{\mathrm{d} t} \phi_t(z) = \Theta_{\phi_t(z)}, \quad \phi_0(z) = z
%\end{equation*}
%\begin{enumerate} 
%\item
%Let $f \cdot \Theta$ be such that $\Theta(f) = 0$. 
%Then the flow $\psi_t$ for $f \cdot \Theta$ is given by
%$\displaystyle \psi_t(z) = \phi_{f(z) \cdot t}(z)$, as can be seen from the following calculation:
%\begin{align*}
%\frac{\mathrm{d}}{\mathrm{d} t} \psi_t(z) = \frac{\mathrm{d}}{\mathrm{d} t} \left( \phi_{f(z) \cdot t}(z) \right) & = f(z) \cdot {\dot{\phi}}_{f(z) \cdot t}(z)\\
%& = f(z) \cdot \Theta_{\phi_{f(z) \cdot t}} \\
%& = f(\phi_{f(z) \cdot t}(z)) \cdot \Theta_{\phi_{f(z) \cdot t}}
%\end{align*}
%Where the last equality follows from $\Theta(f) = 0$, because
%\[
%\frac{\mathrm d}{\mathrm d t} f(\phi_{f(z) \cdot t}(z)) = \mathrm{d}_{\phi_{f(z) \cdot t}} f \circ \dot{\phi}_{f(z) \cdot t}(z) \cdot f(z) = \Theta(f)_{\phi_{f(z) \cdot t}}\cdot f(z) = 0
%\]
%\item
%Now, let $f \cdot \Theta$ be such that $\Theta(f) \neq 0$ but $\Theta^2(f) = 0$.
%Then the flow $\psi_t$ for $f \cdot \Theta$ is given by
%$\displaystyle \psi_t(z) = \phi_{(\exp(t \Theta_z f) - 1) \cdot f(z)/\Theta_z (f)}(z)$. Note that this map is holomorphic in $z$, even for $\Theta_z (f) = 0$.
%The calculation is analogous to the previous case:
We calculate the time derivative of the given $\psi_t$
\begin{align*}
\frac{\mathrm{d}}{\mathrm{d} t} \psi_t(z) &= \frac{\mathrm{d}}{\mathrm{d} t} \left( \phi_{(\exp(t \Theta_z f) - 1) \cdot f(z) / \Theta_z f}(z) \right) \\
& = \Theta_z f \cdot \exp(t \Theta_z f) \cdot f(z) / \Theta_z f  \cdot {\dot{\phi}}_{(\exp(t \Theta_z f) - 1) \cdot f(z) / \Theta_z f}(z)\\
& = f(z) \cdot \exp(t \Theta_z f) \cdot \Theta_{\psi_t(z)}
\end{align*}
and consider now
\begin{align*}
\frac{\mathrm{d}}{\mathrm{d} t} f(\psi_t(z)) &= \mathrm{d}_{\psi_t(z)} f \circ  \dot{\phi}_{(\exp(t \Theta_z f) - 1) \cdot f(z) / \Theta_z f}(z) \cdot f(z) \cdot \exp(t \Theta_z f) \\
&= f(z) \cdot \exp(t \Theta_z f) \cdot \Theta_{\psi_t(z)} f
\end{align*}
Note that $\frac{\mathrm{d}}{\mathrm{d} t} \Theta_{\psi_t(z)} f = 0$ because of $\Theta^2 f = 0$. We can compare the higher order derivatives:
\begin{align}
\label{eqflow1}
\frac{\mathrm{d}^m}{\mathrm{d} t^m} f(\psi_t(z))
&= f(z) \cdot \exp(t \Theta_z f) \cdot(\Theta_z f)^{m-1} \cdot \Theta_{\psi_t(z)} f\\
\label{eqflow2}
\frac{\mathrm{d}^m}{\mathrm{d} t^m} f(z) \cdot \exp(t \Theta_z f)
&=f(z) \cdot \exp(t \Theta_z f) \cdot (\Theta_z f)^m
\end{align}
For $t = 0$ and for all $m \in \NN_0$ the values of \eqref{eqflow1} and \eqref{eqflow2} agree, hence $f(\psi_t(z)) = f(z) \cdot \exp(t \Theta_z f)$ and $\frac{\mathrm{d}}{\mathrm{d} t} \psi_t(z) = f(z) \cdot \Theta_{\psi_t(z)}$.
%\qedhere
%\end{enumerate} 
\end{proof}

We will call the time-$1$ maps of such $\CC$-complete vector fields \emph{ $\Theta$-shears} resp. \emph{$\Theta$-overshears}.

\bigskip
Observe that the action of $\glgrp_n(\CC)$ on $\mat{n}{n}{\CC}$, by conjugation, is not effective, its center is the ineffectivity, and we get an effective action of $\slgrp_n(\CC)$.
Moreover, as a linear representation this is the direct sum of the adjoint representation and a trivial one-dimensional representation, i.e.\ $\mat{n}{n}{\CC} \cong \slalg_n(\CC) \oplus \CC$, where the second summand is the subspace of scalar matrices.
%diagonal matrices with identical entries on the diagonal.

In our context here we will focus on shears and overshears arising from the $\slgrp_n(\CC)$-action on $\Omega_n$ and on $\mat{n}{n}{\CC}$ by conjugation.

A $\Theta$-shear of a vector field $\Theta$ arising from the $\slgrp_n(\CC)$-action will be called a \emph{$\slgrp_n(\CC)$-shear} and a $\Theta$-overshear of such a vector field will be called a \emph{$\slgrp_n(\CC)$-overshear}.

\smallskip
By $E_{a b}$ with $a, b \in \{1, \dots, n\}$ we denote the elementary matrices in $\mat{n}{n}{\CC}$, i.e.\
\begin{equation*}
E_{a b} = \left( \delta_{a k} \delta_{b \ell} \right)_{k, \ell = 1}^n
\end{equation*}
We denote the following commutators as $H_a := [E_{a, a+1}, E_{a+1, a}] = \left( \delta_{a k} \delta_{a \ell} - \delta_{a+1, k} \delta_{a+1, \ell}  \right)_{k, \ell = 1}^n$ for $a = 1, \dots, n-1$. It is well-known that the $E_{a b}$ with $a \neq b$ together with the $H_a$ span the matrix Lie algebra $\slalg_n(\CC)$ as vector space over $\CC$. We need to write down explicitly the adjoint representation of $\slgrp_n(\CC)$ with vector fields and determine the action of these vector fields on polynomials.

For a matrix $V \in \slalg_n(\CC)$ and $X \in \mat{n}{n}{\CC}$ it well known that
\begin{equation*}
\left. \frac{\mathrm d}{\mathrm d t} \exp(t V) \cdot X \cdot \exp(-t V) \right|_{t=0} = [V, X]
\end{equation*}

The entries $x_{k\ell}$ of a matrix $X \in \mat{n}{n}{\CC}$ will serve as coordinates on $\mat{n}{n}{\CC} \cong \slalg_n(\CC) \oplus \CC$. We denote the fundamental vector fields of the adjoint representation of $\slalg_n(\CC)$ corresponding to $E_{ab}$ resp.\ $H_a$ by $\Theta_{ab}$ resp.\ $\Xi_a$. They are 
%\begin{align*}
%[E_{a, b}, X] &= \sum_{\ell = 1}^n \left( \delta_{a k} \delta_{b \ell} \cdot x_{\ell m} - x_{k \ell} \cdot \delta_{a \ell} \delta_{b m} \right)_{k, m} \\
%&= \left( \delta_{a k} \cdot x_{b m} - x_{k a} \delta_{b m} \right)_{k, m} \\
%&= \sum_{k,m = 1}^n \left( \delta_{a k} \cdot x_{b m} - x_{k a} \delta_{b m} \right) \frac{\partial}{\partial x_{k m}} \\
%&= \sum_{k=1}^n \left( x_{b k} \frac{\partial}{\partial x_{a k}} - x_{k a}  \frac{\partial}{\partial x_{k b}} \right)
%\end{align*}
%Similarly, we obtain for $V = H_a$:
%\begin{align*}
%[H_a, X] &= \sum_{\ell = 1}^n \Big( \left( \delta_{a k} \delta_{a \ell} - \delta_{a+1, k} \delta_{a+1, \ell}  \right) \cdot x_{\ell m} \\ &\qquad - x_{k \ell} \cdot \left( \delta_{a \ell} \delta_{a m} - \delta_{a+1, \ell} \delta_{a+1, m}  \right) \Big)_{k, m} \\
%&= \left( \delta_{a k} x_{a m} - \delta_{a+1, k} x_{a+1, m} - x_{k a} \delta_{a m} + x_{k, a+1} \delta_{a+1, m} \right)_{k, m} \\
%&= \sum_{k, m = 1}^n \left( \delta_{a k} x_{a m} - \delta_{a+1, k} x_{a+1, m} - x_{k a} \delta_{a m} + x_{k, a+1} \delta_{a+1, m} \right) \frac{\partial}{\partial x_{k m}} \\
%&= \sum_{k=1}^n \left( x_{a k} \frac{\partial}{\partial x_{a k}} - x_{a+1, k}  \frac{\partial}{\partial x_{a+1, k}} - x_{k a} \frac{\partial}{\partial x_{k a}} + x_{k, a+1} \frac{\partial}{\partial x_{k, a+1}} \right) \\
%\end{align*}
%Hence, we define \nopagebreak
%\begin{remark}
\begin{equation}
\Theta_{a b} := \sum_{k=1}^n \left( x_{b k} \frac{\partial}{\partial x_{a k}} - x_{k a}  \frac{\partial}{\partial x_{k b}} \right), \quad a \neq b
\end{equation}
\begin{equation}
\Xi_a := \sum_{k=1}^n \left( x_{a k} \frac{\partial}{\partial x_{a k}} - x_{a+1, k}  \frac{\partial}{\partial x_{a+1, k}} - x_{k a} \frac{\partial}{\partial x_{k a}} + x_{k, a+1} \frac{\partial}{\partial x_{k, a+1}} \right)
\end{equation}
We will frequently refer to the vector fields $\Xi_a$ as \emph{hyperbolic vector fields}.

The vector fields $\Theta_{a b}$ and their commutators $\Xi_a = [\Theta_{a, a+1}, \Theta_{a+1, a}]$ span the adoint represention of the Lie algebra $\slalg_n(\CC)$ written as vector fields and obey of course the same commutation relations as the $E_{a b}$ and $H_a$. In particular,
\begin{equation}
[\Theta_{a b}, \Theta_{c d}] = 0 \Longleftrightarrow a = c \vee b = d
\end{equation}
%\end{remark} 

\begin{example}
Let $n \geq 2$.
Since $\Theta_{12}(x_{21}) = 0$, the vector field $x_{21} \Theta_{12}$ is a shear vector field. And its flow map is given by
\begin{align*}
X &\mapsto \exp( t x_{21} E_{12}) \cdot X \cdot \exp( -t x_{21} E_{12}) \\
  &= \left( \id + t x_{21} E_{12} \right) \cdot X \cdot \left(\id - t x_{21} E_{12} \right) 
\end{align*}
The semi-group property is satisfied because $x_{21}$ is conjugation invariant under the action of the one-parameter subgroup generated by $\Theta_{12}$.

Now we consider the overshear vector field $x_{11} \Theta_{12}$ with $\Theta_{12}(x_{11}) = x_{21}$ and $\Theta_{12}^2(x_{11}) = 0$.
Using the function $\varepsilon$ from Lemma \ref{lemshearflow}, the flow map is given by
\begin{align*}
X &\mapsto \exp\left( \varepsilon(t x_{21}) \cdot t x_{11} \cdot E_{12} \right) \cdot X \cdot \exp\left( - \varepsilon(t x_{21}) \cdot t x_{11} E_{12} \right) \\
 &= \exp\left( (e^{t x_{21}} - 1) \frac{x_{11}}{x_{21}} \cdot E_{12}\right) \cdot X \cdot \exp\left( - (e^{t x_{21}} - 1) \frac{x_{11}}{x_{21}} \cdot E_{12} \right)
\end{align*}
The semi-group property is less obvious, but can be verified by direct calculation or the more abstract argument in Lemma \ref{lemshearflow}.
\end{example}

\section{Fibred Density Property for the spectral ball}
\label{sectionspectral}

In this section we prove the fibred density property for the spectral ball and determine its automorphism group. The crucial technical part is Proposition \ref{propalldegrees}, and the following lemmas will be needed for the induction in the proof of this proposition.

\begin{lemma}
\label{lem-sl-overshears}
For $\Theta_{a b}, a \neq b$ and $x_{c d}$ with $a, b, c, d \in \{1, \dots n\}$ we have
\[
\Theta_{ab}(x_{c d}) = \delta_{a c} x_{b d} - \delta_{b d} x_{c a}
\]
and for $a < n$ we have
\[
\Xi_a(x_{c d}) = \left( \delta_{ac} - \delta_{a+1,c} - \delta_{ad} + \delta_{a+1, d}\right) x_{c d}
\]
\end{lemma}

\begin{proof}
The proof is a straightforward calculation. \qedhere 
%\begin{align*}
%\Theta_{ab}(x_{c d}) &= \sum_{k=1}^n \left( x_{b k} \frac{\partial}{\partial x_{a k}} - x_{k a}  \frac{\partial}{\partial x_{k b}} \right)(x_{c d}) \\
%&= \sum_{k=1}^n \left( x_{b k} \delta_{ac} \delta_{kd} - x_{k a} \delta_{kc} \delta_{bd} \right) = x_{b d} \delta_{ac} - x_{c a} \delta_{bd} 
%\end{align*}
%\begin{align*}
%\Xi_a(x_{c d}) &= \sum_{k=1}^n \left( x_{a k} \delta_{ac}\delta_{kd} - x_{a+1, k}  \delta_{a+1,c}\delta_{kd} - x_{k a} \delta_{ck}\delta_{ad} + x_{k, a+1} \delta_{ck}\delta_{a+1, d} \right) \\
%&= x_{a d} \delta_{ac} - x_{a+1, d} \delta_{a+1,c} - x_{c a} \delta_{ad} + x_{c, a+1} \delta_{a+1, d} \\
%&= \left( \delta_{ac} - \delta_{a+1,c} - \delta_{ad} + \delta_{a+1, d}\right) x_{c d}
%\qedhere
%\end{align*}
\end{proof}

\begin{cor}
\label{cor-sl-overshears}
\begin{align*}
\Theta_{a b}(x_{cd})   &= 0 &\Longleftrightarrow a \neq c \wedge b \neq d \\
\Theta_{a b}^2(x_{cd}) &= 0 &\Longleftrightarrow a \neq c \vee b \neq d \\
\Theta_{a b}(x_{ab})   &= x_{bb} - x_{aa} \\
\Theta_{a b}^2(x_{ab}) &= -2 x_{ba}, \quad \Theta_{a b}^3(x_{ab}) = 0 \\
\Xi_{a}(x_{cd}) &= 0 &\Longleftrightarrow c = d \vee \{c, d\} \cap \{a, a+1\} = \emptyset
\end{align*}
\end{cor}

We illustrate these results by summarizing them for $\slalg_3(\CC)$ in Tables \ref{tab-sl3fields} and \ref{tab-sl3action}.
\begin{table}[]
\begin{align*}
\Theta_{12} &= + x_{21} \partial_{11} - x_{11} \partial_{12} + x_{22} \partial_{12} + x_{23} \partial_{13} - x_{21} \partial_{22} - x_{31} \partial_{32} \\
\Theta_{13} &= + x_{31} \partial_{11} + x_{32} \partial_{12} - x_{11} \partial_{13} + x_{33} \partial_{13} - x_{21} \partial_{23} - x_{31} \partial_{33} \\
\Theta_{21} &= - x_{12} \partial_{11} + x_{11} \partial_{21} - x_{22} \partial_{21} + x_{12} \partial_{22} + x_{13} \partial_{23} - x_{32} \partial_{31} \\
\Theta_{23} &= - x_{12} \partial_{13} + x_{31} \partial_{21} + x_{32} \partial_{22} - x_{22} \partial_{23} + x_{33} \partial_{23} - x_{32} \partial_{33} \\
\Theta_{31} &= - x_{13} \partial_{11} - x_{23} \partial_{21} + x_{11} \partial_{31} - x_{33} \partial_{31} + x_{12} \partial_{32} + x_{13} \partial_{33} \\
\Theta_{32} &= - x_{13} \partial_{12} - x_{23} \partial_{22} + x_{21} \partial_{31} + x_{22} \partial_{32} - x_{33} \partial_{32} + x_{23} \partial_{33} \\
\Xi_{1} &= +2 x_{12} \partial_{12} + x_{13} \partial_{13} -2 x_{21} \partial_{21} - x_{23} \partial_{23} - x_{31} \partial_{31} + x_{32} \partial_{32} \\
\Xi_{2} &= - x_{12} \partial_{12} + x_{13} \partial_{13} + x_{21} \partial_{21} +2 x_{23} \partial_{23} - x_{31} \partial_{31} -2 x_{32} \partial_{32} 
\end{align*}
\caption{Vector fields for the adjoint representation of $\slalg_3(\CC)$}
\label{tab-sl3fields}
\end{table}

\begin{table}[]
\begin{adjustbox}{max width=\textwidth}
\begin{tabular}{c|cccccccc}
 & $\Theta_{12}$ & $\Theta_{13}$ & $\Theta_{21}$ & $\Theta_{23}$ & $\Theta_{31}$ & $\Theta_{32}$ & $\Xi_{1}$ & $\Xi_{2}$\\
\hline$ x_{11} $ & $ x_{21} $ & $ x_{31} $ & $ - x_{12} $ & $0$ & $ - x_{13} $ & $0$ & $0$ & $0$\\
$ x_{12} $ & $ - x_{11}  + x_{22} $ & $ x_{32} $ & $0$ & $0$ & $0$ & $ - x_{13} $ & $  2x_{12} $ & $ - x_{12} $\\
$ x_{13} $ & $ x_{23} $ & $ - x_{11}  + x_{33} $ & $0$ & $ - x_{12} $ & $0$ & $0$ & $ x_{13} $ & $ x_{13} $\\
$ x_{21} $ & $0$ & $0$ & $ x_{11}  - x_{22} $ & $ x_{31} $ & $ - x_{23} $ & $0$ & $ -2x_{21} $ & $ x_{21} $\\
$ x_{22} $ & $ - x_{21} $ & $0$ & $ x_{12} $ & $ x_{32} $ & $0$ & $ - x_{23} $ & $0$ & $0$\\
$ x_{23} $ & $0$ & $ - x_{21} $ & $ x_{13} $ & $ - x_{22}  + x_{33} $ & $0$ & $0$ & $ - x_{23} $ & $  2x_{23} $\\
$ x_{31} $ & $0$ & $0$ & $ - x_{32} $ & $0$ & $ x_{11}  - x_{33} $ & $ x_{21} $ & $ - x_{31} $ & $ - x_{31} $\\
$ x_{32} $ & $ - x_{31} $ & $0$ & $0$ & $0$ & $ x_{12} $ & $ x_{22}  - x_{33} $ & $ x_{32} $ & $ -2x_{32} $\\
$ x_{33} $ & $0$ & $ - x_{31} $ & $0$ & $ - x_{32} $ & $ x_{13} $ & $ x_{23} $ & $0$ & $0$ \\
\end{tabular}
\end{adjustbox}
\vspace{6pt}
\caption{Action of vector fields on linear monomials for $\slalg_3(\CC)$.}
\label{tab-sl3action}
\end{table}
%
%
%\begin{table}[t]
%\rotatebox{90}{
%\begin{tabular}{c|cccccccc}
%$[\cdot,\cdot]$ & $\Theta_{12}$ & $\Theta_{13}$ & $\Theta_{21}$ & $\Theta_{23}$ & $\Theta_{31}$ & $\Theta_{32}$ & $\Xi_{1}$ & $\Xi_{2}$\\
%\hline
%$\Theta_{12}$ & $0$ & $0$ & $-\Xi_{1}$ & $-\Theta_{13}$ & $\Theta_{32}$ & $0$ & $2\Theta_{12}$ & $-\Theta_{12}$\\
%$\Theta_{13}$ & $0$ & $0$ & $\Theta_{23}$ & $0$ & $-\Xi_{1} - \Xi_{2}$ & $-\Theta_{12}$ & $\Theta_{13}$ & $\Theta_{13}$\\
%$\Theta_{21}$ & $\Xi_{1}$ & $-\Theta_{23}$ & $0$ & $0$ & $0$ & $\Theta_{31}$ & $-2\Theta_{21}$ & $\Theta_{21}$\\
%$\Theta_{23}$ & $\Theta_{13}$ & $0$ & $0$ & $0$ & $-\Theta_{21}$ & $-\Xi_{2}$ & $-\Theta_{23}$ & $2\Theta_{23}$\\
%$\Theta_{31}$ & $-\Theta_{32}$ & $\Xi_{1} + \Xi_{2}$ & $0$ & $\Theta_{21}$ & $0$ & $0$ & $-\Theta_{31}$ & $-\Theta_{31}$\\
%$\Theta_{32}$ & $0$ & $\Theta_{12}$ & $-\Theta_{31}$ & $\Xi_{2}$ & $0$ & $0$ & $\Theta_{32}$ & $-2\Theta_{32}$\\
%$\Xi_{1}$ & $-2\Theta_{12}$ & $-\Theta_{13}$ & $2\Theta_{21}$ & $\Theta_{23}$ & $\Theta_{31}$ & $-\Theta_{32}$ & $0$ & $0$\\
%$\Xi_{2}$ & $\Theta_{12}$ & $-\Theta_{13}$ & $-\Theta_{21}$ & $-2\Theta_{23}$ & $\Theta_{31}$ & $2\Theta_{32}$ & $0$ & $0$\\
%\end{tabular}
%}
%\caption{Commutators of the vector fields for $\slalg_3(\CC)$.}
%\label{tab-sl3commut}
%\end{table}

\begin{lemma}
\label{lem-crossimage}
Let $n \geq 3$. The
$\mathrm{span}_\CC \left\{ \Theta_{ab}(x_{c d}) \,:\, \Theta_{12}(x_{c d}) = 0 \right\}$ contains all linear monomials except $x_{12}$.
\end{lemma}
\begin{proof} \hfill
\begin{enumerate}
\item
We consider first the monomials of the form $x_{aa}$:
Let $x_{a b} \in \ker \Theta_{12}$, i.e.\ $a \neq 1$ and $b \neq 2$.
\[
\Theta_{a b}(x_{a b}) = x_{bb} - x_{aa}
\]
Thus, we obtain all such differences except $x_{22} - x_{11}$, but instead the scalar multiple $x_{11} - x_{22}$. For the coordinate functions on $\slalg_n(\CC)$, i.e.\ elements of $\left(\slalg_n(\CC)\right)^\ast$, we further have the trace condition $x_{11} + x_{22} + \dots + x_{nn} = 0$. Therefore the span contains all functions $x_{aa}$ on  $\slalg_n(\CC)$.

\item For $a \neq 1, d \neq 2$ and $b \neq d$ we obtain
\[
\Theta_{a b}(x_{a d}) = x_{b d}
\]
hence all $x_{k \ell}$ with $\ell \neq 2, k \neq \ell$ are in the span. Here, we need $n \geq 3$.
For $c \neq 1, b \neq 2$ and $c \neq a$ we obtain
\[
\Theta_{a b}(x_{c b}) = - x_{c a}
\]
hence $x_{k \ell}$ with $k \neq 1, k \neq \ell$ are in the span. We again need $n \geq 3$.
\qedhere
\end{enumerate}
\end{proof}

We are now prepared to prove our main proposition. The proof is by induction over the degree. To understand the proof and its structure it might be helpful to look first at the induction step which starts after degree two.

\begin{definition}
\begin{equation*}
\mathcal{L}_n := \mathrm{span}_{\CC} \{ f \cdot \Theta \,:\, f \text{ polynomial on } \slalg_n(\CC), \; \Theta \in \langle \Theta_{k \ell}, \Xi_{m} \rangle \} 
\end{equation*}
By $\mathcal{A}_n$ we denote the Lie algebra generated by all vector fields which are $\slgrp_n(\CC)$-overshears with monomial coefficients of degree at most $2$.
\end{definition}

\begin{proposition}
\label{propalldegrees}
Let $n \geq 2$, then $\mathcal{L}_n = \mathcal{A}_n$.
\end{proposition}

\begin{proof}
We only need to show the inclusion $\mathcal{L}_n \subseteq \mathcal{A}_n$.
The proof is by induction on the degree $d$ of the polynomial coefficients. The induction hypothesis for $d=0$ is true by  assumption.

\bigskip
\textbf{We treat the case $d = 1$ separately:}
The only missing vector fields are $x_{k \ell} \Theta_{k \ell}$ and the linear monomials in front of the hyperbolic vector fields. For $(k,\ell) = (1, 2)$ a short calculation shows:
\[
[x_{22} \Theta_{12}, \Theta_{21}] = x_{22} [\Theta_{12}, \Theta_{21}] - (\Theta_{21}x_{22}) \Theta_{12} = x_{22} \Xi_1 + x_{12} \Theta_{12}
\]
Because $x_{22} \Xi_1$ is a shear vector field, we can conclude that $x_{12} \Theta_{12}$ is in $\mathcal{A}_n$. By symmetry (index permutation in $\Theta_{k\ell}$), this is true for all $x_{k \ell} \Theta_{k \ell}$. 
For the hyperbolic vector fields, see the general case, step \ref{inductionhyper}.

\bigskip
\pagebreak
\textbf{We also need to treat the case $d = 2$ separately:}

\begin{enumerate}
\item
Let $a$ and $f$ be monomials of degree one. We first remark that if $\Theta(a) = 0$ and $\Theta^2(f) = 0$, then $\Theta^2(a f) = 0$, i.e.\ under these assumptions $a f \cdot \Theta$ is a $\Theta$-overshear vector field and hence in $\mathcal{A}_n$. Consider the following difference of Lie brackets:
\begin{align*}
[a f \Theta, \Lambda] - [f \Theta, a \Lambda]
 \quad = \quad & a f [\Theta, \Lambda] - \Lambda(a f) \Theta \\
 & - a f [\Theta, \Lambda] - f \Theta(a) \Lambda + a \Lambda(f) \Theta \\
 = \quad & -f \Lambda(a) \cdot \Theta
\end{align*}
For the quadratic terms in front of $\Theta = \Theta_{k \ell}$ we can restrict ourselves without loss of generality to the case $\Theta = \Theta_{12}$. By Lemma \ref{lem-crossimage} we obtain all terms of the form $f \Lambda(a) \cdot \Theta = x_{ab} x_{cd} \Theta_{12}$ with $x_{ab} \neq x_{12} \neq x_{cd}$.

\item
We focus on the most difficult term, i.e.\ $x_{k \ell}^2 \Theta_{k \ell}$. It is sufficient to consider $\Theta_{12}$:

We make a detour to a hyperbolic vector field and aim to obtain $x_{12}^2 \Xi_1$. Note that $x_{12}^2 \cdot \Theta_{21}$ is a shear vector field.
\begin{align*}
[ x_{12}^2 \cdot \Theta_{21}, \Theta_{12} ]
&= 2 x_{12} \Theta_{12}(x_{12}) \Theta_{21} - x_{12}^2 \Xi_1
\end{align*}
It remains to check that $2 x_{12} \Theta_{12}(x_{12}) \Theta_{21}$ is an overshear vector field:
$\Theta_{21}(2 x_{12} \Theta_{12}(x_{12})) = 2 x_{12} \Theta_{21}(-x_{11}+x_{22}) = 4 x_{12}^2$ and $\Theta_{21}(x_{12}^2)=0$.

Now we calculate the following Lie brackets of already obtained terms:
\begin{align*}
[ x_{12} \cdot \Xi_1, x_{12} \cdot \Theta_{12} ]
&= x_{12} \Xi_1(x_{12}) \Theta_{12} - x_{12} \Theta_{12}(x_{12}) \Xi_1 + x_{12}^2 [\Xi_1, \Theta_{12}] \\
%&= x_{12} \Xi_1(x_{12}) \Theta_{12} - x_{12} \Theta_{12}(x_{12}) \Xi_1 + x_{12}^{2} [\Xi_1, \Theta_{12}] \\
&= 4 x_{12}^{2} \Theta_{12} - x_{12} \Theta_{12}(x_{12}) \Xi_1 \\
[ x_{12}^2 \cdot \Xi_1, \Theta_{12} ]
&= x_{12}^2 [\Xi_1, \Theta_{12}] - 2 x_{12} \Theta_{12}(x_{12}) \Xi_1 \\
&= 2 x_{12}^2 \Theta_{12} - 2 x_{12} \Theta_{12}(x_{12}) \Xi_1
\end{align*}
Now, a linear combination of these Lie brackets yields
\begin{align*}
2 [ x_{12} \cdot \Xi_1, x_{12} \cdot \Theta_{12} ] - [ x_{12}^2 \cdot \Xi_1, \Theta_{12} ] = 6 x_{12}^2 \Theta_{12}
\end{align*}

\item
After having obtained the term $x_{12}^2$ in front of $\Theta_{12}$ we get the other terms by letting $\slalg_n(\CC)$ act on it and subtracting already obtained terms:

\begin{align*}
[x_{12}^2 \Theta_{12}, \Lambda] - [x_{12} \Theta_{12}, x_{12} \Lambda]
 \quad = \quad & x_{12}^2 [\Theta_{12}, \Lambda] - \Lambda(x_{12}^2) \Theta_{12} \\
 & - x_{12} \Theta_{12}(x_{12})\Lambda + x_{12} \Lambda(x_{12}) \Theta_{12} \\
 & - x_{12} x_{12} [\Theta_{12}, \Lambda] \\
\quad = \quad & -x_{12} ( \Lambda(x_{12}) \cdot \Theta_{12} + \Theta_{12}(x_{12}) \cdot \Lambda )
\end{align*}
By Lemma \ref{lem-crossimage} we obtain all terms of the form $x_{12} x_{cd} \Theta_{12}$ if we manage to subtract the terms $x_{12} \Theta_{12} (x_{12}) \cdot \Lambda = - x_{12} x_{11} \Lambda + x_{12} x_{22} \Lambda$.
For this we only need to see that for a $\Lambda = \Theta_{k \ell}, k \neq \ell,$ it is always true that both $\Lambda^2(x_{12}) = 0$ and $\Lambda^2(x_{11}) = 0$ as well as $\Lambda^2(x_{22}) = 0$ which follows from Corollary \ref{cor-sl-overshears}.

\item
For other hyperbolic vector fields, see again the general case, step \ref{inductionhyper}.
\end{enumerate}

\bigskip
\textbf{Induction step} $d \mapsto d + 1, \; d \geq 2$: \nopagebreak
\begin{enumerate}
\item
Let $f$ and $g$ be monomials of degree $d - 1$ (or less) and let $a$ be a monomial of degree one.
\begin{align*}
[ a f \cdot \Theta, \, g \cdot \Lambda ] - [ f \cdot \Theta, \, a g \cdot \Lambda ]
 &= (a f \Theta(g) \Lambda - g \Lambda(a f) \Theta + a f g [ \Theta, \Lambda ]) \\
  &\quad -(f \Theta(a g) \Lambda - a g \Lambda(f) \Theta + a f g [ \Theta, \Lambda ]) \\
 &=  a f \Theta(g) \Lambda - g a \Lambda(f) \Theta - g f \Lambda(a) \Theta \\
  &\quad -a f \Theta(g) \Lambda - f g \Theta(a) \Lambda + a g \Lambda(f) \Theta \\
 &= - f g \cdot \left( \Theta(a) \Lambda + \Lambda(a) \Theta  \right)
\end{align*}
To obtain the coefficients in front of $\Theta_{k \ell}$ it is by symmetry sufficient to consider $\Theta = \Theta_{12}$. Choose $a \in \ker \Theta_{12}$. For $\Lambda$ we can choose any other vector field in $\slalg_n$. From Lemma \ref{lem-crossimage} we know that all linear monomials except $x_{12}$ are obtained as $\Lambda(a)$ in case of dimension $n \geq 3$.
In dimension $n=2$ we only have $\Theta_{12}(x_{21}) = 0$, but -- using also the hyperbolic vector field -- still obtain $\Theta_{21}(x_{21}) = x_{11} - x_{22}$ and $\Xi_1(x_{21}) = 2 x_{21}$; note that $x_{11} + x_{22} = 0$.
We therefore obtain all monomial coefficients $f g \Lambda(a)$ of degree $d+1$ (actually, up to $2d-1$)
 in front of $\Theta_{12}$ except $x_{12}^{d+1}\Theta_{12}$.

\item \label{inductionverynonshear}
To obtain $x_{12}^{d+1} \Theta_{12}$ we calculate:
\begin{align*}
[ x_{12} \cdot \Xi_1, x_{12}^{d} \cdot \Theta_{12} ]
&= x_{12} \Xi_1(x_{12}^{d}) \Theta_{12} - x_{12}^d \Theta_{12}(x_{12}) \Xi_1 + x_{12}^{d+1} [\Xi_1, \Theta_{12}] \\
&= d \cdot x_{12}^{d} \Xi_1(x_{12}) \Theta_{12} - x_{12}^d \Theta_{12}(x_{12}) \Xi_1 + x_{12}^{d+1} [\Xi_1, \Theta_{12}] \\
&= d \cdot x_{12}^{d} \cdot 2x_{12} \Theta_{12} - x_{12}^d (-x_{11} + x_{22}) \Xi_1 + x_{12}^{d+1} 2 \Theta_{12} \\
&= (2 d + 2) x_{12}^{d+1} \Theta_{12} - x_{12}^d (-x_{11} + x_{22}) \Xi_1
\end{align*}
and
\begin{align*}
[ x_{12}^d \cdot \Xi_1, x_{12} \cdot \Theta_{12} ]
&= x_{12}^d \Xi_1(x_{12}) \Theta_{12} - x_{12} \Theta_{12}(x_{12}^d) \Xi_1 + x_{12}^{d+1} [\Xi_1, \Theta_{12}] \\
&= x_{12}^{d} \Xi_1(x_{12}) \Theta_{12} - d \cdot x_{12}^d \Theta_{12}(x_{12}) \Xi_1 + x_{12}^{d+1} [\Xi_1, \Theta_{12}] \\
&= x_{12}^{d} \cdot 2x_{12} \Theta_{12} - d \cdot x_{12}^d (-x_{11} + x_{22}) \Xi_1 + x_{12}^{d+1} 2 \Theta_{12} \\
&= 4 x_{12}^{d+1} \Theta_{12} - d \cdot x_{12}^d (-x_{11} + x_{22}) \Xi_1
\end{align*}
A linear combination of these two Lie brackets yields
\begin{align*}
d \cdot [ x_{12} \cdot \Xi_1, x_{12}^{d} \cdot \Theta_{12} ] - [ x_{12}^d \cdot \Xi_1, x_{12} \cdot \Theta_{12} ] 
&=2 (d^2 + d - 2) x_{12}^{d+1} \Theta_{12}
\end{align*}
and we have found all monomial coefficients of degree $d+1$ in front of the $\Theta_{k \ell}$.

\item \label{inductionhyper}
Now we turn to the hyperbolic vector fields. Again by symmetry it is sufficient to consider $\Xi_1 = [\Theta_{12}, \Theta_{21}]$.
\begin{align*}
[\Theta_{21}, f \cdot \Theta_{12}] = f \cdot \Xi_1 - \Theta_{21}(f) \cdot  \Theta_{12}
\end{align*}
Hence, we obtain all polynomials $f$ of degree $d + 1$ in front of $\Xi_1$ which are such that both $f \cdot \Theta_{12}$ and $\Theta_{21}(f) \cdot  \Theta_{12}$ are already known to be in $\mathcal{A}_n$. Since $\Theta_{21}(f)$ does not increase the degree of $f$, we are done. \qedhere
\end{enumerate}
\end{proof}

\begin{theorem}
\label{fibredforspectral}
The fibration $\pi \colon \mat{n}{n}{\CC} \to \CC^n$ has the fibred density property as has any restriction of $\pi$ to a Runge domain $\Omega \subseteq \mat{n}{n}{\CC}$ with $\pi^{-1}(\pi(\Omega)) = \Omega$.
\end{theorem}

\begin{cor}
\label{fibredforspectralball}
The natural fibration $\pi \colon \sball_n \to \sdisk_n$ has the fibred density property.
\end{cor}

\begin{proof}
The condition $\pi^{-1}(\pi(\sball_n)) = \sball_n$ is clear, because $\sball_n = \pi^{-1}(\sdisk_n)$. The domain $\sball_n \subset \mat{n}{n}{\CC}$ is balanced since obviously $\lambda z \in \sball_n$ for all $z \in \sball_n$ and all $\lambda \in \udisk$. Therefore it is Runge and then by the preceding remark it enjoys the fibred density property.
\end{proof}

We will need the following terminology from representation theory and a result of Dixmier \cite{Dixmier}.

%\begin{definition}
%Let $G$ be a Lie group and $f$ a function on $G$. Then $f$ is called \emph{invariant} if
%\[
%\forall x \in G \; \forall y \in G \quad f(y x y^{-1}) = f(x)
%.\]
%Let $U \subseteq G$ be open subset and $f$ a function defined on $U$. Then $f$ is called \emph{locally invariant} if there exists a neighbourhood $V$ of $\id \in G$ such that 
%\[
%\forall x \in U \; \forall y \in V \quad f(y x y^{-1}) = f(x)
%.\]
%\end{definition}
%
%\begin{definition}
%Let $G$ be a semi-simple complex Lie group of rank $r$. For $x \in G$ denote by $G^x$ the commutant of $x$ in $G$. We call $x$ \emph{regular} if $\dim G^x = r$.
%\end{definition}

\begin{definition}
Let $\mathfrak{g}$ be a complex Lie algebra and $f$ a holomorphic function on $\mathfrak{g}$. Then $f$ is called \emph{invariant} if
\[
\forall x_0 \in \mathfrak{g} \; \forall x \in \mathfrak{g} \quad [x, x_0] f(x) = 0
.\]
Let $U \subseteq \mathfrak{g}$ be an open subset and $f$ a function defined on $U$. A function $f$ is called \emph{locally invariant} if this holds $\forall x_0 \in \mathfrak{g} \; \forall x \in U$.
\end{definition}

%\begin{theorem}[\cite{Dixmier}*{Th\'eor\`eme 2.3}]
%Let $k$ be an algebraically closed field of characteristic $0$ and $G$ be semi-simple algebraic group over $k$, $\mathfrak{g}$ its Lie algebra, $R$ the set of regular elements of $G$, and $X$ a regular vector field on $G$.
%Then the following conditions are equivalent:
%\begin{enumerate}
%\item $X$ annulates the regular invariant functions on $G$.
%\item For all $x \in R$ we have $X(x) \in (\mathop\mathrm{Ad} x - 1)(\mathfrak{g})$.
%\item For all $x \in G$ we have $X(x) \in (\mathop\mathrm{Ad} x - 1)(\mathfrak{g})$.
%\item There exists a regular map $Y : G \to \mathfrak{g}$ such that $X(x) = (\mathop\mathrm{Ad} x - 1) Y(x)$ for all $x \in G$.
%\end{enumerate}
%\end{theorem}
%
%\begin{theorem}[\cite{Dixmier}*{Th\'eor\`eme 2.5}]
%Let $G$ be a semi-simple complex Lie group, $\mathfrak{g}$ its Lie algebra, $R$ the set of regular elements of $G$, $U$ a Stein open subset of $G$, and $X$ a holomorphic vector field on $U$.
%Then the following conditions are equivalent:
%\begin{enumerate}
%\item $X$ annulates the \textbf{locally} invariant functions on $U$.
%\item For all $x \in R \cap U$ we have $X(x) \in (\mathop\mathrm{Ad} x - 1)(\mathfrak{g})$.
%\item For all $x \in U$ we have $X(x) \in (\mathop\mathrm{Ad} x - 1)(\mathfrak{g})$.
%\item There exists a holomorphic map $Y : U \to \mathfrak{g}$ such that $X(x) = (\mathop\mathrm{Ad} x - 1) Y(x)$ for all $x \in U$.
%\end{enumerate}
%\end{theorem}

\begin{theorem}[\cite{Dixmier}*{Th\'eor\`eme 2.4}]
\label{thmDixmier}
Let $\mathfrak{g}$ be a complex semi-simple Lie algebra. Let $U \subseteq \mathfrak{g}$ be an open Stein subset and $\Theta$ a holomorphic vector field on $U$. Then the following conditions are equivalent:
\begin{enumerate}
\item $\Theta$ annihilates the locally invariant functions on $U$.
%\item For all $x \in \mathfrak{r} \cap U$ we have $\Theta(x) \in [x, \mathfrak{g}]$.
%\item For all $x \in U$ we have $\Theta(x) \in [x, \mathfrak{g}]$.
\item There exists a holomorphic map $g \colon U \to \mathfrak{g}$ such that $\Theta(x) = [x, g(x)]$ for all $x \in U$.
\end{enumerate}
\end{theorem}

\begin{cor}
\label{cor-Dix}
Every fibre-preserving holomorphic vector field on an $\Omega \subseteq \mat{n}{n}{\CC}$ which is Stein and satisfies $\Omega = \pi^{-1}(\pi(\Omega))$ can be written as a holomorphic linear combination of the vector fields $\Theta_{ab}, \Xi_c$.
\end{cor}

\begin{proof} Each fibre in $\Omega$ is also a fibre in $\mat{n}{n}{\CC}$ due to $\Omega = \pi^{-1}(\pi(\Omega))$. It follows from Theorem \ref{thmDixmier} that the vector fields $\Theta_{ab}, \Xi_c$ (which form a basis of the Lie algebra)  generate the stalk of the sheaf of germs of fibre-preserving vector fields for the categorical quotient map $\pi\vert_{ \slalg_n(\CC)} \colon \slalg_n(\CC) \to \CC^{n-1}$ at every point, including the singular points. (This result is in fact due to Kostant \cite{Kostant}, and Dixmier's contribution in proving the above Theorem lies in the application of Cartan's Theorem~B.)
If we denote by $\mathcal{T}$ the tangent sheaf and by $\mathcal{T}^\pi$ its subsheaf of sections tangent to $\pi$, this is equivalent to the following exact sequence of sheaves:
\[
\hol_{\slalg_n(\CC)}^{n^2 - 1} \to \mathcal{T}_{\slalg_n(\CC)}^\pi \to 0
\]

Remember that as linear representations $\mat{n}{n}{\CC} \cong \slalg_n(\CC) \oplus \CC$ and thus $\pi \colon \mat{n}{n}{\CC} \to \CC^n $ is equal to $\pi\vert_{ \slalg_n(\CC) } \oplus \id$. 
%Hence the Nakayama lemma shows (note that fibres of a map times identity are still the same fibres of that map) that the vector fields $\Theta_{ab}, \Xi_c$ generate the stalk of the sheaf of germs of fibre-preserving vector fields for $\pi \colon \mat{n}{n}{\CC} \to \CC^n $ at every point.

We tensor the above exact sequence with (the nuclear space of)  holomorphic functions on $\CC$ which corresponds to the trace, and obtain an exact sequence
\[
\hol_{\glalg_n(\CC)}^{n^2 - 1} \to \mathcal{T}_{\glalg_n(\CC)}^\pi \to 0
\]
where we understand the tensor product as a completed tensor product in the sense of Grothendieck. This still yields a right-exact sequence, since we can apply e.g.\ \cite{Treves}*{Proposition 43.9} to small open product sets.

Since $\Omega$ is Stein and the sheaf of germs of fibre-preserving vector fields is coherent, a standard application of Cartan's Theorem~B yields the result. 
\end{proof}

%%%%
%%%%

%By $\vfhol(X)$ we denote the algebraic vector fields on an affine space $X$. We use the completed tensor product $\otimes$ in the sense of Grothendieck.
%We have
%\[
%\vfalg(\glalg_n(\CC)) = \left( \CC[t] \otimes \vfalg(\slalg_n(\CC)) \right) \oplus \left( \vfalg(\CC) \otimes \CC[\slalg_n(\CC)] \right)
%.\]
%For the fibre-preserving vector fields we obtain
%\[
%\vfalg^\pi(\glalg_n(\CC)) = \left( \CC[t] \otimes \vfalg^\pi(\slalg_n(\CC)) \right) \oplus \left( \vfalg^\pi(\CC) \otimes \CC[\slalg_n(\CC)] \right)
%\]
%Let $\Psi_1, \dots, \Psi_{n^2-1}$ be the fibre-preserving vector fields which span the Lie algebra. We have an exact sequence
%\[
%\CC[\slalg_n(\CC)]^{n^2 - 1} \to \vfalg^\pi(\slalg_n(\CC)) \to 0
%\]
%which gives, after tensoring with $\CC[t]$, an exact sequence
%\[
%\CC[\glalg_n(\CC)]^{n^2 - 1} \to \vfalg^\pi(\glalg_n(\CC)) \to 0
%\]

%%%%
%%%%

We are now able to prove the fibred density property for the spectral ball and to determine its automorphism group.

\begin{proof}[Proof of Theorem \ref{fibredforspectral}]
Let $\Theta$ be a fibre-preserving holomorphic vector field on $\Omega_n \subseteq \mat{n}{n}{\CC}$. According to Corollary \ref{cor-Dix} we can write $\Theta$ as a holomorphic linear combination of the vector fields $\Theta_{a b}$ and $\Xi_c$. 
In $\mat{n}{n}{\CC} \cong \CC^{n^2}$ we can approximate these holomorphic linear combinations by polynomial linear combinations. This works as well for any Runge domain $\Omega \subseteq \mat{n}{n}{\CC}$. Note that all holomorphic linear combinations of such vector fields are automatically tangent to the fibres of $\sigma$. By Proposition \ref{propalldegrees} the Lie algebra $\mathcal{L}_n$ is generated by the vector fields in $\mathcal{A}_n$ which correspond to $\slgrp_n(\CC)$-shears and $\slgrp_n(\CC)$-overshears with degree $d \leq 2$. Since we are interested in vector fields with polynomial coefficients not only on $\slalg_n(\CC)$, but also on $\glalg_n(\CC)$, we need to adjoin the trace as additional variable which gives an additional trivial fibre that hence enjoys the fibred density property as well, see Example \ref{trivialexample}.
\end{proof}

\begin{proof}[Proof of Theorem \ref{thmspectraldense}]
\label{proofspectraldense}
Let $f \colon \Omega_n \to \Omega_n$ be a holomorphic automorphism. By \cite{RanWhi1991}*{Thm. 4} we find a M\"obius transformation $h \colon \Omega_n \to \Omega_n$
\begin{equation}
\label{eqmoebius}
A \mapsto \gamma \cdot (A - \alpha \cdot \id) \cdot (\id - \overline{\alpha} A )^{-1}, \quad \alpha \in \udisk, \gamma \in \boundary \udisk
\end{equation}
such that the composition $g := f \circ h^{-1} \colon \Omega_n \to \Omega_n$ is fibre-preserving, i.e.\ $\pi \circ g = \pi$, and in addition $g(0) = 0$.
By \cite{RanWhi1991}*{Thm. 2} we know that under these circumstances $g^\prime(0)$ is a linear automorphism of $\Omega_n$. According to \cite{RanWhi1991}*{Thm.~4, Cor.} it is of the form
\begin{equation*}
A \mapsto G \cdot A \cdot G^{-1} \quad \text{ or } \quad A \mapsto G \cdot A^{t} \cdot G^{-1}, \quad \text{ with } G \in \slgrp_n(\CC)
\end{equation*}
We can connect $g$ to the identity, indeed we connect it to its linear part $g^\prime(0)$ by a $\cont^1$-homotopy $[0, 1] \ni s \mapsto \frac{g(s A)}{s} \to g^\prime(0)(A)$ and then use the fact that any conjugation by a $G \in \slgrp_n(\CC)$ can be connected to the identity or to the matrix transposition.

Therefore we may assume that $f$, after composition with a M\"obius transformation and possibly a transposition is in the identity component of $\faut{\Omega_n}$. By Theorem \ref{fibredforspectral} the spectral ball has the fibred density property and we can apply Theorem \ref{thmfibredAL} with $\mathcal{A} = \mathcal{A}_n$ generated by $\slgrp_n(\CC)$-overshears (with coefficients of at most quadratic degree) as the dense Lie subalgebra generated by complete vector fields.
\end{proof}

%The basic philosophy of the proof is to use the homogeneity of the fibres for proving a fibred density property and the fact that hyperbolicity of the base space forces the automorphisms respect the fibres. Finally, the automorphisms of the base space (i.e.\ the M\"obius transformations) lift. Together they give a full description of a dense subset of the holomorphic automorphism group.

We would like to compare our result to a question asked by 
Ransford and White \cite{RanWhi1991} when they started the study of holomorphic automorphisms of the spectral ball.
They asked whether the automorphisms of the spectral ball are compositions of the M\"obius transformations, the transposition and conjugations of the following form
$X \mapsto \exp(f(X)) \cdot X \cdot \exp(f(X))$ with $f \colon \Omega_n \to \slalg_n(\CC)$ such that it is $\slgrp_n(\CC)$-invariant, i.e.\ $f(G X G^{-1}) = f(X)$ for all $X \in \Omega_n$ and all $G \in \slgrp_n(\CC)$. A counterexample was already given in \cite{Kosinski2012}.

It is however easy to see that the invariance condition is too restrictive: Choose a generic fibre $X_\lambda$. It is a $\slgrp_n(\CC)$-homogeneous complex manifold and hence isomorphic to $\slgrp_n(\CC)/H$ for a reductive subgroup $H$ of $\slgrp_n(\CC)$. Hence there exists a (up to a scalar constant) unique algebraic volume form $\omega^\prime$ on $X_\lambda$ which is invariant under the action of $\slgrp_n(\CC)$ by left-multiplication, see \cite{voldens}*{Appendix}. This volume form $\omega^\prime$ corresponds to a $\slgrp_n(\CC)$-conjugation invariant volume form $\omega$. The invariance condition for the automorphisms of Ransford and White is equivalent to saying that $G(X)$ depends only on $\sigma(X)$. Therefore it necessarily preserves the volume form $\omega$. However, the overshears which are not shears never preserve this volume form, since 
\[
\mathop{\mathrm{div}}( f \Theta) = f \mathop{\mathrm{div}} \Theta + \Theta(f)
.\]
The divergence of a vector field with algebraic flow map (e.g.\ all the $\Theta_{a b}$) necessarily vanishes, and $\Theta(f) \neq 0$ for such overshears. Since all the automorphisms of Ransford and White however have vanishing $\omega$-divergence, no finite or infinite composition of them can yield such an overshear. In the next section we will prove a much stronger result.

\section{Dense, but meagre}
In Theorem \ref{thmspectraldense} we have determined a dense subgroup of $\aut{\sball_n}$. In this section we will prove that this dense subgroup is not the whole $\aut{\sball_n}$, in fact a meagre subset. We follow the original strategy of Anders\'en and Lempert \cite{AL1992}*{Sec. 7}

\begin{definition}
By $P_m$ we denote the vector space of polynomials in $\glalg_n(\CC)^\ast$ of total degree at most $m \in \NN_0$. By $\widetilde{P}_m := P_m \setminus P_{m-1} \cup \{0\}$ we denote the subspace of homogeneous polynomials of total degree $m$.
\end{definition}

The polynomial vector fields $\Theta_{ab}, \; a \neq b,$ and $\Xi_a$ act as $\CC$-linear derivations on the polynomials $\widetilde{P}_m$ into $\widetilde{P}_m$, preserving their total degree. In the following, we want to estimate the dimension of the kernels of $\Theta_{ab} | P_m$ and $\Xi_a | P_m$. By index permutation, it is sufficient to consider $\Theta_{12}$ and $\Xi_1$.

\begin{proposition} We have the following estimates:
\label{prop-growth}
\begin{enumerate}
\item $\ker \left( \Theta_{1 2}^2 | P_m \right)$ grows at most polynomially of degree $n^2 - 1$ in $m$.
\item $\ker \left( \Xi_{1}^2 | P_m \right)$ grows at most polynomially of degree $n^2 - 1$ in $m$.
\end{enumerate}
\end{proposition}

\begin{conjecture}
\label{derivationconjecture}
We conjecture that this estimates holds in more generality:
Let $k$ be a field of characteristic zero. Let $\Theta \colon k[x_1, \dots, x_N] \to  k[x_1, \dots, x_N]$ be a derivation which sends homogeneous polynomials to homogeneous polynomials of the same total degree. By $K_m$ we denote the kernel of $\Theta$ restricted to homogeneous polynomials of degree $m$. Then $\dim K_m$ grows polynomially in $m$ of degree at most $N-2$.
\end{conjecture}

We may choose a new basis such that $\Theta_{12} | \widetilde{P}_1$ is in Jordan normal form. To achieve this, we only need to to make the following change of coordinates: $u := -2 x_{21}, v:= x_{22} - x_{11}, z:= x_{11} + x_{22}$, and all other coordinates remain untouched. We obtain the following Jordan block decomposition:
\begin{itemize}
\item 1 block of size $3 \times 3$ for the basis $x_{12}, u = \Theta_{12}(x_{12}), v = \Theta_{12}(u)$,
\item $2n - 4$ blocks of size $2 \times 2$, each with a basis $x_{1d}, \Theta_{12}(x_{1d}) = x_{2d}, d \neq 1,2$, or $x_{c2}, \Theta_{12}(x_{c2}) = -x_{c1}, c \neq 1,2$.
\item $(n-2)^2+1$ blocks of size $1 \times 1$ corresponding to the remaining monomials and $z$.
\end{itemize}

We first investigate the block of size $3 \times 3$.

\begin{lemma}
\label{growth-three}
Let $\Theta = x \frac{\partial}{\partial w} + y \frac{\partial}{\partial x}$ act on $\CC[w,x,y]$.

Then $\dim \ker ( \Theta | \widetilde{P}_m ) \leq 3m$ for all $m \in \NN$.
\end{lemma}
\begin{proof}
We write a homogeneous polynomial of degree $m$ as a polynomial in $w$ with coefficients in $x$ and $y$.
\[
\begin{split}
0 = \Theta \left( \sum_{0 \leq k, \ell \leq m}^m c_{k,\ell} \cdot w^{m-k-\ell} x^k y^\ell \right) \\
= \sum_{0 \leq k, \ell \leq m}^m \big( &c_{k,\ell} \cdot (m-k-\ell) \cdot w^{m-k-\ell-1} x^{k+1} y^\ell
\\  &+ c_{k,\ell} \cdot k \cdot w^{m-k-\ell} x^{k-1} y^{\ell+1}  \big)
\end{split}
\]
We can read off the condition for the coefficients:
\begin{equation}
\forall k, \ell \quad : \quad c_{k+1, \ell-1} = -\frac{m-k-\ell+1}{k+1} c_{k-1, \ell}
\end{equation}
Organizing the coefficients $c_{k,\ell}$ in a matrix, we see that the first two rows and the first line together completely determine all the other coefficients, hence $3 m$ is an upper bound. 
\end{proof}

\begin{remark}
This estimate can be improved, since actually a lot of the coefficients $c_{k \ell}$ have to be zero. However, for our application it is enough to have just some linear growth.
\end{remark}

\begin{lemma}
\label{growth-induction}
Let $\Theta = y \frac{\partial}{\partial x}$ act on $A := \CC[x, y]$ and let $\Psi$ act on $B := \CC[u_1, \dots, u_N]$ as a derivation which sends homogeneous polynomials to homogeneous polynomials of the same total degree. By $K_m$ we denote the kernel of $\Theta + \Psi$ acting on $A \otimes B$ and consisting of homogeneous polynomials of degree $m$. Moreover we assume that the dimension of the kernel of $\Psi$, restricted to homogeneous polynomials of total degree $m$, is of polynomial growth of degree $N-2$ in $m \in \NN$.

Then, $\dim K_m$ grows polynomially of degree $N$ in $m \in \NN$, \\ and $\sum_{\tilde{m} = 0}^m \dim K_{\tilde{m}}$ grows polynomially of degree $N+1$ in $m \in \NN$.
\end{lemma}

\begin{proof}
Let denote $d_m$ the dimension of the kernel of $\Psi$ when restricted to homogeneous polynomials of total degree $m$. Restricting $\Psi$ to this finite dimensional vector, we obtain a nilpotent endomorphism. From its Jordan normal form we can read off the estimate for the kernel of $\Psi^\ell$, which is simply given by $(1+\ell) d_m$.

We decompose the elements of $A \otimes B$ as polynomials in $x$ and $y$ with coefficients in B:
\begin{align*}
(\Theta + \Psi)\left( \sum_{k, \ell \in \NN_0} x^k y^\ell q_{k, \ell} \right) = \sum_{k, \ell \in \NN_0} \left( k x^{k-1} y^{\ell+1} q_{k, \ell} + x^k y^\ell \Psi(q_{k,\ell}) \right) = 0
\end{align*}
This polynomial in $x$ and $y$ vanishes if and only if 
\begin{equation}
\forall k, \ell \in \NN_0 \quad : \quad \Psi( q_{k, \ell} ) + (k + 1) \cdot q_{k+1, \ell-1} = 0
\end{equation}
where we assume the coefficients $q$ to vanish if an index is negative.

Up to rational coefficients, this can be visualized in a matrix as follows:
\[
\begin{pmatrix}
q_{0, 1}       & q_{0, 1}        & q_{0, 2} & \dots & q_{0, m} \\
\Psi(q_{0, 1}) & \Psi(q_{0, 2})  &          &       & q_{1, m} \\
\Psi^2(q_{0, 2})         &          &          &       &  \vdots \\
 \vdots        &          &          &       &  \vdots \\
\Psi^m(q_{0, m}) &          &     &    & q_{m, m} \\         
\end{pmatrix}
\]

The conditions can be subsumed in the following way:
\begin{itemize}
\item The coefficient $q_{k, 0}$ has to be in the kernel of $\Psi$.
\item Hence, the coefficient $q_{0, \ell}$ has to be in the kernel of $\Psi^{\ell+1}$ and of total degree $m - \ell$ in $u$ and $v$.
\item The coefficients $q_{k,\ell}$ with $k + \ell \geq m + 1$ must vanish due to too high degree.
\item All other coefficients are completely determined by the above. There are no other conditions.
\end{itemize}

We sum up the dimensions for the choices of $q_{k, 0}, 0 \leq k \leq m$:
\[
\dim K_m = \sum_{k=0}^m (1+k) \cdot d_{m-k}
\]
and obtain a polynomial in $m$ of degree $N$.
\end{proof}

\begin{lemma}
\label{lem-Xi-growth}
\begin{equation}
\begin{split}
\dim \ker ( \Xi_1 | \widetilde{P}_m ) = \sum_{\stackrel{\scriptstyle 2(p - q) + (k - \ell) = 0 }{0 \leq s := p + q + k + \ell \leq m} }  \binom{4n-4 + k - 1}{k} \cdot \\ \cdot \binom{4n-4 + k - 1}{\ell} \cdot \binom{(n-2)^2 + 2 + m - s -1}{m - s}
\end{split}
\end{equation}
\end{lemma}
\begin{proof}
All linear monomials are eigenfunctions of $\Xi_1$; the following eigenvalues can occur:
\begin{itemize}
\item $\pm 2$ each once,
\item $\pm 1$ each with multiplicity $2(n-2)$,
\item $0$ with multiplicity $(n-2)^2+2$
\end{itemize}
By the product rule, all monomials are eigenfunctions of $\Xi_1$, and their eigenvalues are just the sums of the eigenvalues of their linear factors. Hence, the dimension of the kernel of $\Xi_1$ is the number of linearly independent monomials with eigenvalue zero. If we choose $p$ resp.\ $q$ times a factor with eigenvalue $\pm 2$, and $k$ resp.\ $\ell$ times a factor with eigenvalue $\pm 1$, we have to satisfy the following conditions: $2(p - q) + (k - \ell) = 0$ and $0 \leq s := p + q + k + \ell \leq m$. It then remains to choose $(m - s)$ times a factor with eigenvalue $0$.
\end{proof}

\begin{proof}[Proof of Proposition \ref{prop-growth}]
As argued in Lemma \ref{growth-induction} we have the estimates $\dim \ker ( \Theta_{12}^2 | \widetilde{P}_m ) \leq 2 \dim \ker ( \Theta_{12} | \widetilde{P}_m )$. Moreover, we have $\dim \ker ( \Xi_{1}^2 | \widetilde{P}_m ) = \dim \ker ( \Xi_{1} | \widetilde{P}_m )$ since all monomials are eigenfunctions of $\Xi_{1}$.

We prove the first statement by induction on the Jordan blocks of $\Theta_{12}|\widetilde{P}_1$.
The induction starts with Lemma \ref{growth-three} on the single block of size $3 \times 3$. We proceed using Lemma \ref{growth-induction} for adjoining all the blocks of size $2 \times 2$. The growth of $\dim K_m$ is always a polynomial in $N-2$ when the basis consists of $N$ elements. Adjoining all the remaining zero-blocks of size $1 \times 1$ is just tensoring with a ring where the derivation acts trivially, hence just an additional choice of free variables, each increasing the degree by $1$. For $N=n^2$ we end up with the desired estimate that $\dim K_m$ grows of degree $n^2 - 2$ in $m$.

The second statement follows directly from Lemma \ref{lem-Xi-growth}. The largest term in the sum is of degree $(n-2)^2 + 1$ in $m$, and there are at most $m^3$ such terms, hence the polynomial growth is of order $(n-2)^2 + 1 + 3 = n^2 - 4n + 8 \leq n^2 - 2$ for $n \geq 3$. In the case $n=2$ we note that there are no linear factors with eigenvalues $\pm 1$ present, hence only $m$ terms can appear in the sum, and $(n-2)^2 + 1 + 1 = n^2 - 2$ for $n=2$.
\end{proof}

\begin{remark} \hfill
\begin{enumerate}
\item
If you compare Proposition \ref{prop-growth} to \cite{AL1992}*{Sec.~7} and their notation, note that $\ker \left( \Theta_{a b}^2 | P_m \right)$ resp.\ $\ker \left( \Xi_{a}^2 | P_m \right)$ correspond to the vector space of polynomials $\Theta_0^m \times \Theta^{m-1}$ in their notation.
\item
Note one subtlety about the degree: For a $p \in \ker \left( \Theta_{a b} | P_m  \right) \subset \ker \left( \Theta_{a b}^2 | P_m  \right)$ the polynomial vector field $p \cdot \Theta_{a b}$ is of the degree $m+1$, but the corresponding matrix $\exp(p \cdot E_{a b})$ will contain polynomials of degree $m$, since $E_{a b}^2 = 0$ for $a \neq b$. However, the polynomial shear automorphism arising by conjugation with this matrix will be of degree $2m+1$.
\end{enumerate}
\end{remark}

It is convenient to introduce the notion of density property also for Lie algebras.

\begin{definition}[\cite{Varolin2}*{Definition 0.1}]
Let $\mathfrak g$ be a Lie algebra of holomorphic vector fields on a complex manifold $X$. We say that $\mathfrak g$ has the \emph{density property} if the Lie subalgebra of $\mathfrak g$ generated by the complete holomorphic vector fields is dense in $\mathfrak g$.
\end{definition}

\begin{definition}
Following \cite{Varolin2}*{Section 0} we recall the definition of jet spaces:
\begin{enumerate} 
\item
Let $J^k_{\mathfrak{g}}(X)$ be the space of $k$-jets of local biholomorphisms of the form $\varphi_{t_m}^{\Theta_m} \circ \dots \circ \varphi_{t_1}^{\Theta_1}$ where $\varphi_t^\Theta$ is the local flow map at time $t \in \CC$ of the vector field $\Theta$ for vector fields $\Theta_1, \dots, \Theta_m \in \mathfrak{g}$ and times $t_1, \dots, t_m$ small enough. For a jet $\gamma \in J^k_{\mathfrak{g}}(X)$ we denote by $\sigma(\gamma)$ its source point and denote
$J^{k,x}_{\mathfrak{g}} = \{\gamma \in J^{k}_{\mathfrak{g}} \,:\, \sigma(\gamma) = x \}$.
\item
By $\mathrm{Aut}_{\mathfrak{g}}(X)$ we denote the subgroup of $\aut{X}$ generated by the time-$1$-maps of flows of complete holomorphic vector fields.
\item For a holomorphic map $f \colon X \to X$, its $k$-jet at $x \in X$ is denoted by $j^k_x(f)$.
\end{enumerate}
\end{definition}

\begin{theorem}[\cite{Varolin2}*{Theorem 0.1}]
\label{jetinterpol}
Let $\mathfrak{g}$ be the Lie algebra of holomorphic vector fields on a complex manifold $X$ with the density property. Then for each jet $\gamma \in J^k_{\mathfrak{g}}(X)$ there exists $f \in \mathrm{Aut}_{\mathfrak{g}}(X)$ such that
\[
j^k_{\sigma(\gamma)}(f) = \gamma
\]
\end{theorem}

\begin{remark}
\label{jetremark}
From the proof of Theorem 0.1 in \cite{Varolin2}*{Section 3} we can in fact deduce a slightly stronger statement, namely that for each jet $\gamma \in J^k_{\mathfrak{g}}(X)$ there exists an open neighborhood $U$ in $J^k_{\mathfrak{g}}(X)$ and a continuous section of $j^k$ on $U$. Hence, the map $j^k$ is open.
\end{remark}

\begin{lemma}
Let $\mathfrak{g}$ be the Lie algebra of all fibre-preserving vector fields of $\pi \colon \sball_n \to \sdisk_n$. The jet space $J^{2m,0}_{\mathfrak{g}}(\sball_n)$  is a finite-dimensional complex-affine space and
\begin{align}
\dim J^{2m+1,0}_{\mathfrak{g}}(\sball_n) &\geq \binom{m + n^2}{n^2}
\end{align}
\end{lemma}

\begin{proof}
For this estimate, it is enough to consider all matrix conjugations of the form 
\begin{align*}
X &\mapsto \overbrace{(\id + a(X) \cdot E_{12})}^{=: A(X)} \cdot X \overbrace{\cdot (\id - a(X) \cdot E_{12})}^{= A^{-1}(X)} \\
&= X + a(X) \cdot [E_{12}, X] - a^2(X) \cdot E_{12} X E_{12}
\end{align*}
where $a$ is a polynomial of degree at most $m$ in $\glalg(\CC)^\ast$. The vector space of these polynomials has dimension $\binom{m+n^2}{m} = \sum_{k=0}^m \binom{n^2 + m - 1}{m}$. Each such polynomial gives rise to a different conjugation (consider e.g.\ the matrix entry $(2,2)$) which in turn defines a $(2m + 1)$-jet that preserves the fibres of $\pi$ and is locally invertible. The map $\phi(X) := A(X) \cdot X \cdot A^{-1}(X)$ can be connected to its linear part by a path $1/s \cdot \phi(s \cdot X)$ and letting $s \to 0$, and can then be connected further to the identity, as shown already in the proof of Theorem \ref{thmspectraldense} on page \pageref{proofspectraldense}. By the fibred density property, the vector field can then be approximated by linear combinations and Lie brackets of complete vector fields, hence the jet of $\phi$ indeed belongs to $J_{\mathfrak{g}}^{2m+1,0}(\sball_n)$. 
\end{proof}

\begin{cor}
\label{cormeagreestimate}
For fixed $k \in \NN$ and $m \in \NN$ large enough:
\[
\dim J^{2m+1,0}_{\mathfrak{g}}(\sball_n) \geq k \cdot \max\left\{ \dim \ker \left( \Theta_{a b}^2 | P_m \right), \dim \ker \left( \Xi_{a}^2 | P_m \right) \right\}
\]
\end{cor}

\begin{proof}
The l.h.s.\ is a polynomial in $m$ of degree $n^2$ and the r.h.s.\ is of degree at most $n^2 - 1$ according to Proposition \ref{prop-growth}.
\end{proof}

\begin{definition}
For the vector fields, we define a truncation map $\mathrm{tr}_m \colon \hol(\sball_n) \to P_m$ which sends a holomorphic function to its Taylor polynomial about $0$ of degree $m$.
\end{definition}

\begin{lemma}
\label{lemtruncationcommut}
The following diagram commutes:
\[
\begin{diagram}
\node{\ker \left( \Lambda_{1}^2 \right) \times \dots \times \ker \left( \Lambda_{k}^2 \right)} \arrow{e,t}{\Psi_k} \arrow{s,l}{\mathrm{tr}_m \times \dots \times \mathrm{tr}_m} \node{\faut{\sball_n}} \arrow{s,r}{j_0^{2m+1}}\\
\node{\ker \left( \Lambda_{1}^2 | P_m \right) \times \dots \times \ker \left( \Lambda_{k}^2 | P_m \right)} \arrow{e,t}{\psi_k} \node{ J^{2m+1,0}_{\mathfrak{g}}(\sball_n)}
\end{diagram}
\]
where $\Lambda_{1}, \dots, \Lambda_{k} \in \left\{ \Theta_{a b} \,:\, 1 \leq a \neq b \leq n\right\} \cup \left\{ \Xi_a \,:\, 1 \leq a \leq n-1 \right\}$ are acting as derivations on $\hol(\sball_n)$, and 
\begin{align*}
\Psi_k(f_1, \dots f_k) &:= \exp(f_1 \Lambda_{a_1}) \cdots \exp(f_k \Lambda_{a_k}) \\
\psi_k &:= j_0^{2m+1} \circ \Psi_k.
\end{align*}
\end{lemma}
\begin{proof}
We first need to show that the maps really map into the given targets. This is clear for $\Psi_k$. Since the $\Lambda_1^2, \dots, \Lambda_k^2$ preserve the total degree of polynomials, the truncation $\mathrm{tr}_m$ actually maps the kernel into itself.
The diagram commutes by definition of $\psi_k$.
\end{proof}

\begin{proof}[Proof of Theorem \ref{thmmeagre}]
We can restrict ourselves to the case of the subgroup $\faut{\sball_n}$ of fibre-preserving automorphisms, since $\aut{\sball_n}$ is generated by $\faut{\sball_n}$ together with M\"obius transformations and matrix transposition. We follow the idea of the proof of \cite{AL1992}*{Theorem 7.1}.

The topology on $\faut{\sball_n}$ shall be the topology of local uniform convergence for both the automorphisms and their inverses, which is a completely metrizable space. 
We denote by $C_k, \; k \in \NN,$ the set of automorphisms obtained by the composition of $k$ overshears of $\Theta_{a b}$ resp. $\Xi_a$. Using the notation of Lemma \ref{lemtruncationcommut} above, we see that $C_k$ is the image of the map $\Psi_k$. We  claim that the set $C_k$ is meagre in $\faut{\sball_n}$ for all $k \in \NN$. By Baire's  theorem it would then follow that $\cup_{k \in \NN} C_k \neq \faut{\sball_n}$ since $\cup_{k \in \NN} C_k$ would be meagre too. 

Assume now by contradiction that $C_k$ is non-meagre in $\faut{\sball_n}$ for some $k \in \NN$. Then we set $V := (\Psi_k)^{-1}(\faut{\sball_n})$ and $V_m := (\mathrm{tr}_m \times \dots \times \mathrm{tr}_m)(V)$. Since $j_{0}^{2m+1}$ is an open mapping by Remark \ref{jetremark}, also $j_{0}^{2m+1}(\Psi_k(V))$ is non-meagre in $J_{\mathfrak{g}}^{0, 2m+1}(\sball_n)$. By Lemma \ref{lemtruncationcommut} also $\psi_k(V_m)$ is non-meagre in $J_{\mathfrak{g}}^{0, 2m+1}(\sball_n)$. However, the mapping $\psi_k$ is differentiable. Then inequality of Corollary \ref{cormeagreestimate} forces $\psi_k(V_m)$ to be meagre, a contradiction.
\end{proof}

\end{document}